\documentclass{amsart}

\usepackage{graphicx} % Required for inserting images
\usepackage{slashbox}
\usepackage{tikz}
\usetikzlibrary{calc}
\usepackage{graphicx}
\usepackage{booktabs}
\usepackage{amssymb}
\usepackage{amsmath}
\usepackage{amsthm, amsfonts, mathrsfs}

\newtheorem{theorem}{Theorem}
\newtheorem{lemma}[theorem]{Lemma}
\newtheorem{cor}[theorem]{Corollary}

\newtheorem{ques}{Question}

\newtheorem{conj}{Conjecture}

\title{Exploring Homological Properties of Independent Complexes of Kneser Graphs}
\date{\today}
\author[Z. Feng]{Ziqin Feng}
\address{Department of Mathematics and Statistics\\Auburn University\\Auburn\\AL~36849\\USA}
\email{zzf0006@auburn.edu}

\author[G. Wang]{Guanghui Wang}
\address{School of Mathematics\\Shandong University\\Jinan\\China}
\email{ghwang@sdu.edu.cn}


\begin{document}

\maketitle
\begin{abstract}
We discuss the topological properties of the independence complex of Kneser graphs, Ind(KG$(n, k))$, with $n\geq 3$ and $k\geq 1$. By identifying one kind of maximal simplices through projective planes, we obtain homology generators for the $6$-dimensional homology of the complex Ind(KG$(3, k))$. Using cross-polytopal generators, we provide lower bounds for the rank of $p$-dimensional homology of the complex Ind(KG$(n, k))$  where $p=1/2\cdot {2n+k\choose 2n}$.  

Let $\mathcal{F}_n^{[m]}$ be the collection of $n$-subsets of $[m]$ equipped with the symmetric difference metric. We prove that if $\ell$ is the minimal integer with the $q$th dimensional reduced homology $\tilde{H}_q(\mathcal{VR}(\mathcal{F}^{[\ell]}_n; 2(n-1)))$ being non-trivial, then $$\text{rank} (\tilde{H}_q(\mathcal{VR}(\mathcal{F}_n^{[m]}; 2(n-1)))\geq \sum_{i=\ell}^m{i-2\choose \ell-2}\cdot \text{rank} (\tilde{H}_q(\mathcal{VR}(\mathcal{F}_n^{[\ell]}; 2(n-1))). $$      Since the independence complex Ind(KG$(n, k))$ and the Vietoris-Rips complex $\mathcal{VR}(\mathcal{F}^{[2n+k]}_n; 2(n-1))$ are the same, we obtain a homology propagation result in the setting of independence complexes of Kneser graphs. Connectivity of these complexes is also discussed in this paper.   
    
\end{abstract}

\section{Introduction}
We denote $[m]$ to be the set of integers $\{1, 2, \ldots, m\}$ for each integer $m\geq 1$. Fix integers $n\geq 2$ and $k\geq 0$. The Kneser graph, KG$(n, k)$ is the graph with vertices being the collection of all $n$-subsets of $[2n+k]$ and any pair of vertices being adjacent if they have empty intersection; the stable Kneser graph, SG$(n, k)$ is the graph with vertices being the collection of all $n$-subsets of $[2n+k]$ not containing pairs $\{i, i+1\}$ or $\{1, 2n+k\}$ and any pair of vertices being adjacent if they have empty intersection. Kneser conjectured in 1955 that the chromatic number of KG$(n, k)$ is $k+2$.    

Lov\'{a}sz in \cite{LL78} proved that Kneser's conjecture holds using the Bosuk-Ulam theorem. Shrijver in \cite{SC78} proved that $\chi(\text{SG}(n, k))= \chi(\text{KG}(n, k))$, again using the Bosuk-Ulam theorem; moreover, he showed that if $G$ is  a subgraph of $\text{SG}(n, k)$ obtained by removing vertices, then  $\chi(G)<\chi(\text{SG}(n, k))$. This means that the stable Kneser graphs are vertex critical. These results and proof techniques used are one of the resources which lead to the work involving the interaction of combinatorics and algebraic topology.    

An independence complex is a construction from a graph which unveils the topological interplay among its independent sets. For any graph $G=(V, E)$, the independence complex of $G$, Ind$(G)$, is the simplicial complex with the simplices being the independent sets in the graph $G$. Barmak in \cite{Bar13} obtained a lower bound of the chromatic number of any graph using the topological property of the independence complex, in fact he proved that  $$\chi(G)\geq \text{Cat(Ind(}G))+1.$$
For any topological space $X$, the strong category number, Cat($X$),  is the minimum integer number $n$ such that there exists a CW-complex $Y$ homotopy equivalent to $X$ which can be covered by $n + 1$ contractible subcomplexes;  if such an $n$ doesn't exist, Cat($X$) is infinite. Lots of work has been done to understand the topology of independence complexes of graphs, for example \cite{EH06, BLN08, E08, B22}. A generalized version, $r$-independence complex, is introduced and studied in \cite{ADGRS23, DSS22}. 

  Barmak in \cite{Bar13} proved that Ind(KG($2, k$)) is homotopy equivalent to the wedge sum of ${k+3\choose 3}$-many copies of $S^2$. Computational results in Table~\ref{KG_3_homology} show that there are at least two non-trivial homologies in the complex, Ind(KG($3, k$)). Braun in \cite{BB09} proved that Ind(SG($2, k$)) is also a wedge sum of spheres when $k\geq 4$. The complex Ind(KG($n, 0$)) is a boundary of cross-polytope with ${2n\choose n}$-many vertices and hence it is homotopy equivalent to a sphere with dimension $\frac{1}{2}{2n\choose n}-1$. A detailed discussion of Ind(KG($n, 0$)) can be found in Section~\ref{max_cpgenerators}.  Much remains unknown regarding the topological properties of these independence complexes. The purpose of this paper is to investigate  the independence complexes of Kneser graph, Ind(KG($n, k$)), with $n\geq 3$ and $k\geq 1$. 
  
  There is a natural connection as discussed in \cite{FN24} between Ind(KG($n, k$)) and certain Vietoris-Rips complexes.  The Vietoris-Rips complex $\mathcal{VR}(X;r)$ of a metric space $(X,d)$ with scale $r\geq 0$ is a simplicial complex with vertex set $X$, where a nonempty finite subset $\sigma$ of $X$ is a simplex in $\mathcal{VR}(X;r)$ if and only if $d(x, y)\leq r$ for any pair $x, y\in \sigma$.   Along with the development of topological data analysis \cite{RG08, GC09}, it is very important to determine the topological properties of Vietoris-Rips complex of finite metric spaces in applied topology. In fact, the idea behind persistent homology is to compute the (co)homology of a Vietoris-Rips complex filtration built on data, which is typically a finite metric space in high dimensions (\cite{UB21}).
 Vietoris-Rips complexes were introduced by Vietoris in \cite{VI27} and then by Rips in \cite{EG87} to approximate a metric space at a chosen scale for different purposes.  A lot of attention has been drawn to study the homotopy types of Vietoris-Rips complexes of different metric spaces, for example, circles and ellipses (\cite{AA17, AAR19}), metric graphs (\cite{GGPSWWZ18}), geodesic spaces (\cite{ZV19, ZV20}), planar point sets (\cite{ACJS18}, \cite{CDEG10}), hypercube graphs (\cite{AA22, F23, FN24,  Shu22}), and more (\cite{MA13}, \cite{AAGGPS20}, \cite{AFK17}).%Additionally, these kinds of complexes have been intensively used in computational topology as a simplical model for the sensor networks (\cite{GM05, SG06, SG07}) and as a tool for image processing (\cite{SMC07}).

In this paper, we consider the following metric space. For $m\geq n$, let $S$ be a subset of $[m]$ with $n\leq |S|$. We define $\mathcal{F}_n^{S}$ to be a collection of $n$-subsets of $S$ equipped with a metric $d$ such that, for any $A$ and $B$ in $\mathcal{F}_n^{[m]}$, $d(A, B)=|A\Delta B|$, where $A\Delta B$ denotes the symmetric difference of $A$ and $B$, i.e., $(A\setminus B)\cup (B\setminus A)$. Notice as in Lemma~\ref{dist} that for any $A$ and $B$ in $\mathcal{F}_n^{[m]}$, $d(A, B)\leq 2(n-1)$ if and only if $A\cap B\neq \emptyset$; then, it is straightforward to verify that the independence  complex Ind(KG$(n, k))$ is same as the Vietoris-Rips complex $\mathcal{VR}(\mathcal{F}_n^{[2n+k]}; 2(n-1))$. In general, it is interesting to investigate the complex $\mathcal{VR}(\mathcal{F}_n^{[2n+k]}; r)$ with a general scale $r$. The results in \cite{FN24} show that the complex $\mathcal{VR}(\mathcal{F}_n^{[2n+k]}; 2)$ is a wedge sum of $S^2$'s for $n\geq 2$ and $k\geq 0$. 

We start with some preliminaries in Section~\ref{intro}. In Section~\ref{max_cpgenerators}, we discuss two important classes of maximal simplices (Lemma~\ref{max_pp} and Lemma~\ref{max_2n}) in the independence complex, Ind(KG($n, k))$. One of two classes is obtained through finite projective planes. When $n=3$, the maximal simplex from the projective plane of order $2$, namely the Fano plane, can be naturally extended as a cross-polytopal homology generator which gives a non-trivial homology in the corresponding dimension (Theorem~\ref{n3_nontrivial_hom}). Using the cross-polytopal generator induced by another class of maximal simplices, we identify a lower bound for the rank of $p$-dimensional homology of Ind(KG($n, k))$ with $p=\frac{1}{2}\cdot {2n\choose n}-1$ for general $n$ (Theorem~\ref{low_b_bigdim}). The lower bounds obtained match the computed results of $9$-dimensional homology in Table~\ref{KG_3_homology} for $n=3$ and $k=1, 2, 3$.  In Section~\ref{centration}, we introduce the concentration maps and prove their properties (Lemma~\ref{phi_m_S} and Lemma~\ref{phi_m_m-1}) which allow us to build the homology propagation results (Theorem~\ref{low_b_smalldim}) in Section~\ref{homology_p}. The connectivity of the independence complex is discussed (Theorem~\ref{connectivity}) in Section~\ref{connect} and we conclude with a list of open questions in Section~\ref{open_p}.    

 \begin{table}
 \caption{Ranks of non-trivial homologies in the independence complex of the Kneser graph KG($3, k$) computed through Auburn University Easley Cluster.}
\label{KG_3_homology}
\begin{tabular}{ |c|p{0.8cm}|p{0.8cm}|p{0.8cm}|p{0.8cm}| } 
 \hline
 \backslashbox[2mm]{homology}{$k$} & $0$ & $1$ & $2$& $3$   \\ 
  \hline
 $6$th-dim & $0$ &  29  & 233 & 1,052 \\
 \hline
 $9$th-dim & $1$ & $7$ &  $28$  & $84$  \\
 \hline
 \end{tabular}
\end{table}

\section{Notation and Preliminaries}\label{intro}

\textbf{Topological spaces.} Let $X$ and $Y$ be topological spaces. We write $X\simeq Y$ when they are homotopy equivalent. %We denote $S^k$ to be the $k$-dimensional sphere. The wedge sum of $X$ and $Y$, $X\vee Y$, is the space obtained by gluing $X$ and $Y$ together at a single point. The homotopy type of $X\vee Y$ is independent of the choice of points if $X$ and $Y$ are connected CW complexes. For $k\ge 1$, $\bigvee_k X$ denotes the $k$-fold wedge sum of $X$.  
We denote $\Sigma X$ to be the suspension of $X$. For any sphere $S^d$, the suspension $\Sigma S^d$ is homeomorphic to $S^{d+1}$. 

Any two metric spaces $(X, d_X)$ and $(Y, d_Y)$ are said to be isometric if there is a bijective distance-preserving map $f$ from $X$ to $Y$, i.e., $d_X(x_1, x_2)=d_Y(f(x_1), f(x_2))$ for any $x_1, x_2\in X$. Hence if $X$ and $Y$ are isometric, then it is straightforward to verify that  $\mathcal{VR}(X, r)$ is homeomorphic to $\mathcal{VR}(Y, r)$ for any $r\geq 0$.

A cross-polytope with $2d$ vertices is a regular, convex polytope that exists in $d$-dimensional Euclidean space. So it homeomorphic to the unit ball in $\mathbb{R}^d$ whose boundary is homeomorphic to $S^{d-1}$.

\medskip

\textbf{Simplicial complexes.} A simplicial complex $K$ on a vertex set $V$ is a collection of non-empty subsets of $V$ such that: i) all singletons are in $K$; and ii) if $\sigma\in K$ and $\tau\subset \sigma$, then $\tau\in K$. For a complex $K$, we use $K^{(k)}$ to represent the $k$-skeleton of $K$, which is a subcomplex of $K$. For vertices $v_1, v_2, \ldots, v_k$ in a complex $K$,  if they span a simplex in $K$, then we denote the simplex to be $\{v_1, v_2, \ldots, v_k\}$. If $\sigma$ and $\tau$ are simplices in $K$ with $\sigma\subset \tau$, we say $\sigma$ is a face of $\tau$.  We say a simplex is a maximal simplex (or a facet) if it is not a face of any other simplex.

A complex $K$ is \emph{a clique complex} if  the following condition holds: a non-empty subset $\sigma$ of vertices is in $K$ given that the edge $\{v, w\}$ is in  $K$ for any pair $v, w\in \sigma$. For any graph $G=(V, E)$, we denote Cl$(G)$ to be the clique complex of $G$ whose vertex set is $V$ and Cl$(G)$ contains a finite subset $\sigma\subset V$ as a simplex if each pair of vertices in $\sigma$ forms an edge in $G$. Also, we see that the Vietoris-Rips complex over any metric space is a clique complex by the definition.

Let $L$ be a complex and $v$ be a vertex not in $L$. The cone over $L$ with the vertex $v$, denoted by $v\ast L$, is a simplical complex defined on the vertex set $L^{(0)}\cup \{v\}$ such that a simplex of $v\ast L$ is either a simplex in $L$ or a simplex in $L$ adjoined with $v$. Notice that any cone is contractible.

For any vertex $v$ in a complex $K$, $K\setminus v$ denote the induced complex on the vertex set $K^{(0)}\setminus \{v\}$. The star of a vertex $v$ in $K$ is st$_K(v)=\{\sigma: \sigma\cup \{v\}\in K\}$. Hence for any $v\in V$, st$_K(v)$ is contractible because it is a cone over $\text{lk}_K(v)$ with the vertex $v$, namely $v\ast \text{lk}_K(v)$, where $\text{lk}_K(v)=\{\sigma: \sigma\cup \{v\}\in K\text{ and }v\notin \sigma\}$.  The next lemma gives  an important method to investigate the homotopy type of a complex by splitting it into a vertex and its complement. A proof could be found in \cite{FN24}. 

\begin{lemma} \label{complex_add_1v}If $v$ is a vertex in $K$ with the inclusion map $\imath: \text{lk}_K(v)\rightarrow K$ being null-homotopic, then $K$ is homotopic to $K\setminus v\vee \Sigma (\text{lk}_K(v))$. \end{lemma}

The following lemma is straightforward to prove and it is convenient to use in discussing the relation between Vietoris-Rips complexes and independence complexes. 

\begin{lemma}\label{dist} Let $A$ and $B$ be $n$-subsets of $[m]$. For any integer $c = 1, 2, \ldots,  n$, $d(A, B)\leq 2c$ if and only if $|A\cap B| \geq n-c$. Specifically $d(A, B)\leq 2(n-1)$ if and only if $A\cap B\neq \emptyset$. \end{lemma}

\textbf{Cross-polytopal homology generators.} To identify generators of the homology of general simplicial complexes, we consider the clique complex of a specific class of graphs. Let $G=(V, E)$ be the graph with vertex set $V=\{v_1, v_2, \ldots, v_{2d}\}$ and $\{v_i, v_j\}\in E$ if and only if $|j-i|\neq d$. Let Cl$(G)$ be the clique complex of $G$ consisting all the clique subsets of $V$, i.e. pairwise connected subsets of $V$.  Then the complex Cl$(G)$ is the boundary of a cross-polytope with $2d$ vertices which is homotopy to $S^{d-1}$. And this complex has a $(d-1)$-dimensional cycle $\gamma$ which generates $H_{d-1}(\text{Cl}(G))$. We say this kind of complex to be \emph{cross-polytopal}. Let $\mathcal{A}$ be the collection of maximal antipode-free simplices as $$\mathcal{A}=\{\sigma\subset V: v_i\in \sigma \text{ iff } v_j \notin \sigma \text{ for } \{v_i, v_j\} \in E \}.$$ 
Then the cycle $\gamma$ is the sum of properly oriented elements in the collection $\mathcal{A}$.

Next we aim to identify two sufficient conditions for a subcomplex $L$ of $K$ under which the inclusion map from $L$ to $K$ induces an injective homology homomorphism.  The next lemma is proved by Adams and Virk in \cite{AV24}. 

\begin{lemma}\label{same_coef_max} Suppose $K$ is a simplicial complex and $\sigma$ is a maximal simplex of dimension $p$ in $K$. If there is a $p$-cycle $\alpha$ in $K$ in which $\sigma$ appears with a non-trivial coefficient $\lambda$, then any representative $p$-cycle of $[\alpha]$ contains $\sigma$ with the same coefficient $\lambda$. \end{lemma}

Using the result above, we obtain the first sufficient condition for the inclusion map of a subcomplex $L$ of $K$ to induce an injective homology homomorphism. 

\begin{lemma}\label{cross_p_injective} Let $L$ be a cross-polytopal subcomplex of the complex $K$ with $2d$ vertices and $\mathcal{A}$ be the collection of maximal antipode-free simplices in the subcomplex $L$.

If there exists an element in $\mathcal{A}$ that is a maximal simplex in $K$, then the inclusion map $\imath: L\rightarrow K$ induces an injective map in the $(d-1)$-dimensional homology and hence the $(d-1)$-dimensional homology of $K$ is non-trivial.  \end{lemma}

\begin{proof} Choose a $(d-1)$-cycle $\gamma$ which is a sum of properly oriented elements in $\mathcal{A}$. Let $\sigma$ be a simplex in $\mathcal{A}$ which is a maximal simplex in $K$. If the map $\imath_\ast$ induced by the inclusion map $\imath$ is not injective, then there is a non-zero number $\lambda$ such that $\imath_\ast(\lambda\cdot[\gamma])=0$. By Lemma~\ref{same_coef_max}, the coefficient of $\sigma$ in  any representative of $[\lambda\cdot \gamma]$  is $\lambda$. Then there is a $d$-chain in $K$ whose boundary is a representative of $[\lambda\cdot \gamma]$. This is impossible since $\sigma$ is maximal in $K$. \end{proof}

The next result is proved by Virk in \cite{Virk22} which states that if a subcomplex $L$ to be a contraction of the complex $K$, then the inclusion map induces injective homology homomorphisms. Recall that a map $f:X\rightarrow Y$ from a metric space $(X, d)$ onto a closed subspace $Y\subseteq X$ is a \emph{contraction} if $f|_Y = \text{id}_Y$ and $d(f(x), f(y))\leq d(x, y)$ for all $x, y\in X$. It is straightforward to verify that if $f$ is a contraction from $X$ to $Y$ and $g$ is a contraction from $Y$ to $Z$ with $Y$ being a closed subspace of $X$ and $Z$ being a closed subspace of $Y$, then $g\circ f$ is a contraction from $X$ to $Z$. In Section~\ref{centration}, we define contraction map from $\mathcal{F}^{[m]}_n$ to $\mathcal{F}^{S}_n$ with $S\subseteq [m]$ and this map allows us to develop homology propagation results using the following lemma in the setting of $\mathcal{F}^{[m]}_n$.

\begin{lemma}\label{contraction_ind_inj} If $f:X\rightarrow Y$ is a contraction, then the inclusion map $Y\rightarrow X$ induces injections on  homology $H_q(\mathcal{VR}(Y; r))\rightarrow H_q(\mathcal{VR}(X; r))$ for all integers $q\geq 0$ and scales $r\geq 0$. 
\end{lemma}

\section{Maximal simplices and Cross-polytopal generators}\label{max_cpgenerators}

In this section, we seek the cross-polytopal subcomplexes in the complex Ind($KG(n, k))$ which contain a maximal simplex and these subcomplexes hence yields non-trivial homologies in Ind($KG(n, k))$, together with lower bounds for these homologies. We start with discussing the maximal simplices in  Ind($KG(n, k))$.

\textbf{Maximal simplices in Ind($KG(n, k))$.}  As we discussed in the introduction, $$\text{Ind}(KG(n, k))=\mathcal{VR}(\mathcal{F}_n^{[2n+k]}; 2c)$$ with $c= n-1$. So instead we investigate some maximal simplices in $\mathcal{VR}(\mathcal{F}_n^{[m]}; 2c)$ with $m\geq 2n$ and $c=n-1$. One obvious maximal simplex in such complex is $\{A: i_0\in A \text{ and }A\in \mathcal{F}_n^{[m]}\}$ for some $i_0\in [m]$. There are two kinds of  maximal simplices playing crucial roles in determining the generators of the non-trivial homology of these complexes. Next we'll give their definitions and proofs. We start with one obvious kind of maximal simplices.  

\begin{lemma} Let $S$ be a subset of $[m]$ with size $2n-1$. Then $\mathcal{F}_n^{S}$ forms a maximal face in $\mathcal{VR}(\mathcal{F}_n^{[m]};2c)$ with $c=n-1$. \end{lemma}

The proof is straightforward if we notice that: i) any pair of vertices in $\mathcal{F}_n^{S}$ has non-empty intersection and hence their distance is $\leq 2(n-1)$ by Lemma~\ref{dist}; ii) if $A$ is an $n$-sized subset of $[m]$ and $A\setminus S\neq \emptyset$, then $|S\setminus A|\geq n$ and pick an $n$-subset $B$ of $S\setminus A$ which has empty intersection with $A$, hence $d(A, B)>2(n-1)$. 

There is one kind of maximal simplex in $\mathcal{VR}(\mathcal{F}_n^{[m]}; 2(n-1))$ that is related to projective planes. A projective plane of order $n$ where $n \geq  2$ is a finite set of points and lines (defined as sets of points), such that:

\begin{itemize}

\item[i)] every line contains $n + 1 $ points;
\item[ii)] every point lies on $n + 1$ lines;

\item[iii)] any two distinct lines intersect in a unique point;

\item[iv)] any two distinct points lie on a unique line. 
\end{itemize}

It is showed (Lemma 4 in \cite{XP16}) that any projective plane of order $n$ has $(n^2+n+1)$-many points and lines, which is a Steiner system $S(2, n+1, n^2+n+1)$ considering the lines to be blocks. Any finite projective plane can also be considered a uniform hypergraph with edges to be the lines in the projective plane. For example, the Fano plane is a projective plane of order $2$, which is also a $3$-regular and $3$-uniform hypergraph with $7$ vertices and $7$ edges.  One hypergraph representation of Fano plane (see the picture below) is $(V, E)$ with $V=[7]$ and $$E=\{\{1, 2, 3\}, \{1, 4, 5\}, \{1, 6, 7\}, \{2, 4, 6\}, \{2, 5, 7\}, \{3, 4, 7\}, \{3, 5, 6\}\}.
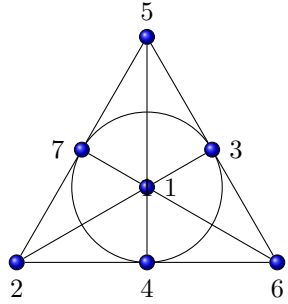
$$
\begin{center}
\begin{tikzpicture}
\tikzstyle{point}=[ball color=blue, circle, draw=black, inner sep=0.07cm]
\node[label=right: 1] (v7) at (0,0) [point] {};
\draw (0,0) circle (1cm);
\node[label=above: 5] (v1) at (90:2cm) [point] {};
\node[label=below: 2] (v2) at (210:2cm) [point] {};
\node[label=below: 6] (v4) at (330:2cm) [point] {};
\node[label=left: 7] (v3) at (150:1cm) [point] {};
\node[label=below: 4] (v6) at (270:1cm) [point] {};
\node[label=right: 3] (v5) at (30:1cm) [point] {};
\draw (v1) -- (v3) -- (v2);
\draw (v2) -- (v6) -- (v4);
\draw (v4) -- (v5) -- (v1);
\draw (v3) -- (v7) -- (v4);
\draw (v5) -- (v7) -- (v2);
\draw (v6) -- (v7) -- (v1);
\node at (v7) {1};
\end{tikzpicture}
\end{center}
It is known that a projective plane of order $n$ exists if $n$ is a prime power. Also, a projective plane of order $6$ or $10$ doesn't exists (see \cite{CL91}). It is an open question for the existence of projective planes of other orders. 

A blocking set of a projective plane is a subset of points which meets all the lines but doesn't contain any line. It is proved in \cite{BS78} that any blocking set of a projective plane with order $n$ has at least size $k$ such that 
$$k\geq n+\sqrt{n}+1.$$
%$$n+\sqrt{n}+1\leq k\leq n^2-\sqrt{n}.$$ 

%\begin{lemma} Let $(\mathcal{V}, \mathcal{E})$ be an $(\ell+1)$-uniform hypergraph with $\mathcal{V}=[\ell^2+\ell+1]$ and $\ell\geq 2$ such that it is a projective plane of order $\ell$ if considering $\mathcal{E}$ be the collection of the lines. Then for any $(\ell+1)$-sized subset $A$ of $\mathcal{V}$, either $A\in \mathcal{E}$ or there is an edge $B$ in $\mathcal{E}$ with $B\subset \mathcal{V}\setminus A $.   \end{lemma}
%\begin{proof} Let $A$ be an $(\ell+1)$-subset of $\mathcal{V}$ and $A\notin \mathcal{E}$. Suppose, for a contradiction, that any $(\ell+1)$-subset of $\mathcal{V}\setminus A$ is not in $\mathcal{E}$. Hence $1\leq |A\cap B|\leq \ell$ for each $B\in \mathcal{E}$. \end{proof}

Next we show that if the projective plane exists at order $n$ with points being the set $[n^2+n+1]$, then the collection of lines in this projective plane forms a maximal simplex in the complex $\mathcal{VR}(\mathcal{F}_{n+1}^{[m]}; 2n)$ with $m\geq n^2+n+1$. 

\begin{lemma}\label{max_pp}  Let $n$ be an integer such that the projective plane of order $n$ exists and $m\geq n^2+n+1$. Let $G=(V, E)$ be an $(n+1)$-uniform hypergraph with $V=[n^2+n+1]$ and $E\subset \mathcal{F}^V_{n+1}$ which forms a projective plane of order $n$. 

Then the collection $E$ forms a maximal simplex in the complex $\mathcal{VR}(\mathcal{F}^{[m]}_{n+1}, 2n)$. \end{lemma}

\begin{proof} First, notice the collection $E$ of edges in the hypergraph $G$ is a subset of $\mathcal{F}^{[m]}_{n+1}$. Hence the elements in $E$ are also the vertices in the complex  $\mathcal{VR}(\mathcal{F}^{[m]}_{n+1}, 2n)$. Then $E$ forms a simplex in $\mathcal{VR}(\mathcal{F}^{[m]}_{n+1}, 2n)$ because the intersection of any pair of the vertices in $E$ contains exactly one element in $[n^2+n+1]$, i.e. their distance is $2n$ by Lemma~\ref{dist}. 

Suppose, for contradiction, that the simplex $E$ is not maximal in the complex $\mathcal{VR}(\mathcal{F}^{[m]}_{n+1}, 2n)$. Then we take a vertex $B\in \mathcal{F}_{n+1}^{[m]}$ such that $B\notin E$ and $d(B, A)\leq 2n$ for each vertex $A\in E$. Hence by Lemma~\ref{dist}, $B\cap A\neq \emptyset$ for each vertex $A\in E$. Then $B$ is a blocking set of the projective plane. Note that $B$ has size $n+1$. This is impossible because all the blocking sets of a projective plane with order $n$ have size $\geq n+\sqrt{n}+1$ by the result in \cite{BS78}.  This contradiction shows that $E$ forms a maximal simplex in the complex $\mathcal{VR}(\mathcal{F}_{n+1}^{[m]}, 2n)$. \end{proof}

We start with one of the maximal simplices obtained from Fano plane in the complex $\mathcal{VR}(\mathcal{F}_3^{[m]}, 4)$ with $m\geq 7$ and extend it to be a cross-polytopal generator in $6$-dimensional homology. And this shows that the $6$-dimensional homology of these complexes is non-trivial. It is worth to mention that the approach in the following result fails for a maximal simplex obtained from the projective plane of order $n\geq 3$.

\begin{theorem} \label{n3_nontrivial_hom} Let $m$ be an integer $\geq 7$ and $$\sigma = \{\{1, 2, 3\}, \{1, 4, 5\}, \{1, 6, 7\}, \{2, 4, 6\}, \{2, 5, 7\}, \{3, 4, 7\}, \{3, 5, 6\}\}.$$ 
Then the following hold. 

\begin{itemize}

\item[1)] The collection $\sigma$ forms a maximal simplex in the complex $\mathcal{VR}(\mathcal{F}_3^{[m]}, 4)$.

\item[2)] Fix $A\in \sigma$ and $A'$ being a $3$-subset  of $[7]\setminus A$. Then  $d(A', B)\leq 4$ for any $B\in \sigma$ with $B\neq A$. 

\item[3)] Fix $S$ be a $6$-subset in $[7]$ and for each $A\in \sigma$ define $\psi_S(A)$ to be a $3$-subset of $S\setminus A$. Then the complex $\mathcal{VR}(\sigma\cup \psi_S(\sigma); 4)$ is cross-polytopal with $14$ vertices and hence it is homotopy equivalent to $S^6$.  

\item[4)] The homology group $H_6(\mathcal{VR}(\mathcal{F}_3^{[m]}; 4))$ is nontrivial for $m\geq 7$. 
\end{itemize}
\end{theorem}
\begin{proof} For part 1), the collection $\sigma$ forms an maximal simplex in $\mathcal{VR}(\mathcal{F}_3^{[m]}, 4)$ by Lemma~\ref{max_pp} since it is the collection of lines in a projective plane of order $2$ with points be $[7]$. 

To prove part 2), we fix $A\in \sigma$ and an arbitrary $3$-subset $A'$ of $[7]\setminus A$. Pick $B\in \sigma$ such that $B\neq A$. Then by the property of projective plane, $|B\cap A|=1$; hence the size of the intersection $B\cap([7]\setminus A)$ is $2$. Because the size of the complement of $A$ in $[7]$ is $4$, $B$ has a nonempty intersection with any $3$-subset of $[7]\setminus A$. This implies that $d(A', B)\leq 4$ by Lemma~\ref{dist}. 

Next, we'll prove part 3). By part 2), it is sufficient to show that $d(\psi_S(A), \psi_S(B)) \leq 4$ for any pair $A, B\in \sigma$. By Lemma~\ref{dist}, we need to show that $\psi_S(A)\cap \psi_S(B)\neq \emptyset$ for any pair $A, B\in \sigma$. Fix $A, B\in \sigma$. Notice that $|A\cap B|=1$. We divide the proof into the following three cases. 
\begin{itemize}
\item[i)] Suppose that $A, B\subset S$. Then $\psi_S(A)=S\setminus A$ and $\psi_S(B)=S\setminus B$. Since $|A\cap B|=1$, the size of the intersection of $A$ with the complement  $B$ in $S$, $\psi_S(B)$,  is $2$.  Therefore the size of the intersection of $\psi_S(A)$ with $\psi_S(B)$ is $1$, i.e.,  $|\psi_S(A)\cap \psi_S(B)|=1$. 

\item[ii)] Suppose that $A\subseteq S$ and $[7]\setminus S\subset B$. Note that $|A\cap B|=1$ and $A\cap B\in S$. Then $|A\cap (S\setminus B)|=2$, hence $|\psi_S(A) \cap (S\setminus B)|= 2$. Since $|S\setminus B|=4$, $|\psi_S(A) \cap \psi_S(B)|= 1$.    

\item[iii)] Suppose that $A\cap B\in [7]\setminus S$. Then the size of the complement of $A\cup B$ in $S$ is  $2$ which means that $|(S\setminus A)\cap (S\setminus B)|=|S\setminus (A\cup B)| = 2$. Note that $|S\setminus B|= |S\setminus A|=4$. Then, $|\psi_S(A)\cap \psi_S(B)| \geq 1$. 
\end{itemize}
Hence the complex $\mathcal{VR}(\sigma\cup \psi_S(\sigma); 4)$ is cross-polytopal. This finishes the proof of part 3).

For part 4), we fix $m\geq 7$ and $S=[6]$. As in part 3) we define $\psi_S(A)$ for each $A\in \sigma$. Consider the complex  $L=\mathcal{VR}(\sigma\cup \psi_S(\sigma); 4)$ which is cross-polytopal by part 3). The simplex $\sigma$ is in the collection of collection of maximal antipode-free simplices of $L$ and $\sigma$ is a maximal simplex in $\mathcal{VR}(\mathcal{F}_3^{[m]}, 4)$. Therefore by Lemma~\ref{cross_p_injective}, the inclusion map from $L$ to $\mathcal{VR}(\mathcal{F}_3^{[m]}, 4)$ induces an injective homology homomorphism. Since $L$ is homotopy equivalent to $S^6$, the $6$-dimensional homology of $\mathcal{VR}(\mathcal{F}_3^{[m]}, 4)$ is nontrivial. 
\end{proof}

Fix $m\geq 2n$ and $p=\frac{1}{2}{2n\choose n}-1$. Next, we use the same strategy to prove that the $p$-dimensional homology of the complex $\mathcal{VR}(\mathcal{F}^{[m]}_n; 2(n-1))$ is nontrivial, and also we obtain lower bounds of the rank of $p$-dimensional homology group.  Fix $S=\{s_1, s_2, \ldots, s_{2n}\}$ being a subset of $[m]$. We consider an isometric embedding $\jmath_S: \mathcal{F}_n^{[2n]}\hookrightarrow \mathcal{F}_n^{[m]}$ induced by the natural bijective correspondence $f(i)=s_i$ from $[2n]$ to $S$. Notice that $\jmath_S(\{i_1, \ldots, i_n\})=\{s_{i_1}, \ldots, s_{i_n}\}$ for $\{i_1, \ldots, i_n\}\in \mathcal{F}_n^{[2n]}$ which clearly induces a simplicial map from $\mathcal{VR}(\mathcal{F}_n^{[2n]}; 2(n-1))$ to $\mathcal{VR}(\mathcal{F}_n^{[m]}; 2(n-1))$ and we use the same notation for the induced simplical map.  We show that the induced homology homomorphism  $(\jmath_S)_\ast$  is injective on $p$-dimensional homology; and furthermore, we obtain a lower bound on the rank of $p$-dimensional homology on  $\mathcal{VR}(\mathcal{F}_n^{[m]}; 2(n-1))$ for general $m$ which is hence also a lower bound on the rank of $p$-dimensional homology on the independence complex, Ind($KG(n, k)$), when $2n+k=m$. 

It is straightforward to determine the homotopy type of the complex  $\mathcal{VR}(\mathcal{F}_n^{[2n]}; 2(n-1))$. Observe that any vertex in the complex  $\mathcal{VR}(\mathcal{F}_n^{[2n]}; 2(n-1))$ is adjacent to all other vertices but its complement in $[2n]$. Hence the clique complex  $\mathcal{VR}(\mathcal{F}_n^{[2n]}; 2(n-1))$ is cross-polytopal and homotopy equivalent to a $p$-dimensional sphere by the discussion in Section~\ref{intro}. We denote the collection of maximal antipode-free simplices in the complex $\mathcal{VR}(\mathcal{F}_n^{[2n]}; 2(n-1))$ by $\mathcal{A}_n$.  We show in Lemma~\ref{max_2n} that at least one element in the collection $\mathcal{A}_n$ is maximal in the complex $\mathcal{VR}(\mathcal{F}_n^{[m]}; 2(n-1))$ for any $m\geq 2n$. 

For each simplex $\sigma$ in the complex $\mathcal{VR}(\mathcal{F}_n^{[m]}; 2(n-1))$ for $m\geq n$, we denote cHull($\sigma$) to the convex hull of $\sigma$, $\mathcal{VR}(\mathcal{F}_n^S; 2(n-1))$, where $S=\bigcup\{A: A\in \sigma\}$.

\begin{lemma}\label{max_2n} Let $m\geq 2n$ with $n\geq 3$ and $S$ be a $2n$-subset of $[m]$ listed by $\{i_1, i_2, \ldots, i_{2n}\}$. 

Then, there exists a simplex $\sigma$ in $\mathcal{A}_n$ such that $\jmath_S(\sigma)$ is a maximal simplex in $\mathcal{VR}(\mathcal{F}_n^{[m]}; 2(n-1))$ and cHull($\jmath_S(\sigma))=\jmath_S (\mathcal{F}_{n}^{[2n]})$. \end{lemma}

\begin{proof} To define $\sigma$, we start with all the $n$-subsets of $[2n-1]$, denoted by $\tau$. Then we define $\sigma$ to the collection obtained by  replacing $\{1, 2, \ldots, n\}$ and $\{1, n+1, \ldots, 2n-1\}$ in $\tau$ by their complements in $[2n]$, i.e., $\{n+1, n+2, \ldots, 2n\}$ and $\{2, 3, \ldots, n, 2n\}$; hence $\sigma$ is antipodal-free and maximal. So $\sigma$ is in $\mathcal{A}_n$ and it is a maximal simplex in $\mathcal{VR}(\mathcal{F}_n^{[2n]}, 2(n-1))$.

We claim that $\sigma$ is a maximal simplex in the complex $\mathcal{VR}(\mathcal{F}_{n}^{[m]}; 2(n-1))$.  This implies that $\jmath_S(\sigma)$ is maximal in $\mathcal{F}_n^{[m]}$ due to the fact that $\jmath_S$ is isometric. 

For convenience, we denote $c=n-1$. To prove the claim using contradiction, suppose that $\sigma$ is not maximal in the complex $\mathcal{VR}(\mathcal{F}_{n}^{[m]}; 2c)$. Then there is an $B\in \mathcal{F}_{n}^{m}$ with $\sigma\cup \{B\}$ is still a simplex in the complex $\mathcal{VR}(\mathcal{F}_n^{m}; 2c)$. This implies that $B\cap A\neq \emptyset$ for each $A\in \sigma$.  Since $\sigma$ is maximal in $\mathcal{F}_n^{[2n]}$, $B$ is not $\mathcal{F}_n^{[2n]}$ which means that $B \setminus [2n]\neq \emptyset$. Without loss of generality, assume that the size of the set  $B\setminus [2n]$ is $1$. And we denote $B\setminus [2n]=\{j_0\}$ and $B=\{j_0, i_1, \ldots, i_{n-1}\}$ where $i_1, i_2, \ldots, i_{n-1}$ are in $[2n]$. Let $T=[2n]\setminus B$ which contains $n+1$ elements. Notice that $B\cap \{n+1, n+2, \ldots, 2n\}\neq \emptyset$ and $B\cap \{2, 3,  \ldots, n, 2n\}\neq \emptyset$. Then  we divide the proof in the following two cases. 

\begin{itemize}
\item[i)] Suppose $2n\in B$. Then $2n\notin T$. Choose an $n$-subset $D\subset [2n-1]$ of $T$ which is different from $\{1, 2, \ldots, n\}$ and $\{1, n+1, \ldots, 2n-1\}$. Note that $T$ is consisting of $n+1$ numbers $\leq 2n-1$ and $T$ has $(n+1)$-many $n$-subsets all of which are in $\sigma$ except for  $\{1, 2, \ldots, n\}$ and $\{1, n+1, \ldots, 2n-1\}$.  Hence such a $D$ exists. So  $D$ is an element in  $\sigma$, but $D\cap B=\emptyset$. Therefore, $d(B, D)= 2n$ which is a contradiction.   

\item[ii)] Now we suppose $2n\notin B$. Since $B\cap \{n+1, n+2, \ldots, 2n\}\neq \emptyset$, there  a $j_1\in B$ with $n+1\leq j_1\leq 2n-1$. And since $B\cap \{2, 3,  \ldots, n, 2n\}\neq \emptyset$,  there a $j_2\in B$ with $2\leq j_2\leq n$. Also, notice that $2n$ is in $T$. Then let $D=T\setminus \{2n\}$; and then $D$ is an $n$-subset of $[2n-1]$. Notice that $j_1, j_2\notin T$; this implies that $D$ is an $n$-subset of $[2n-1]$ different from $\{1, 2, \ldots, n\}$ and $\{1, n+1, \ldots, 2n-1\}$. Hence $D\in \sigma$ and $B\cap D=\emptyset$ which is a contradiction.  \end{itemize}
Hence $\imath_S(\sigma)$ is a maximal simplex in the complex $\mathcal{VR}(\mathcal{F}_{n}^{[m]}; 2c)$.   

Note that $S=\bigcup \jmath_S(\sigma)$ since $\bigcup \sigma=[2n]$. Hence, the convex hull of $\imath_S(\sigma)$ is the subcomplex $\mathcal{VR}(\mathcal{F}_{n}^{S}; 2c)$.  \end{proof}

Using the maximal simplicies from Lemma~\ref{max_2n}, we prove a lower bound for the rank of $p$-dimensional homology group of the complex $\mathcal{VR}(\mathcal{F}_n^{[m]}; 2(n-1))$ for all $m\geq 2n$. Examples of different $n$ and $k$ are given in Table~\ref{tab:largest_dim_rank}. Note that ranks of the $9$-dimensional homology when $n=3$ and $k=1, 2, 3$ match with the computed results in Table~\ref{KG_3_homology}. 

\begin{theorem}\label{low_b_bigdim}  Suppose that $n\geq 3$, $m\geq 2n$ and $p = \frac{1}{2}{2n\choose n}-1$. Then   

$$\text{rank}(H_p(\mathcal{VR}(\mathcal{F}_n^{[m]}; 2(n-1))))\geq {m\choose 2n}.$$

\end{theorem}
\begin{proof} There are ${m\choose 2n}$-many subsets of $[m]$ and we list them as $S_1, S_2, \ldots, S_{m\choose 2n}$. For each $S_i$, we can fix a bijective map from $[2n]$ to $S_i$ which induces an  isometric embedding of the metric space $\mathcal{F}_n^{[2n]}$ into $\mathcal{F}_n^{[m]}$. Hence there are ${m\choose 2n}$-many natural embeddings of $\mathcal{F}_n^{[2n]}$ in $\mathcal{F}_n^{[m]}$.  

Fix $i=1, 2, \ldots, {m\choose 2n}$. Let $\sigma_i$ be the maximal simplex in the complex $\mathcal{VR}(S_i;2(n-1))$ which satisfies the conditions in Lemma~\ref{max_2n}, i.e., $\sigma_i$ is also a maximal simplex in $\mathcal{VR}(\mathcal{F}_n^{[m]};2(n-1))$  and its convex hull is $\mathcal{VR}(\mathcal{F}^{S_i}_n; 2(n-1))$. Let $\alpha_i$ be one representation of the cross-polytopal generator of $H_p(\mathcal{VR}(S_i; 2(n-1)))$ such that the coefficient of $\sigma_i$ is $1$. And let $\omega_i$ be the $p$-cochain on the complex $\mathcal{VR}(\mathcal{F}_n^{[m]}; 2(n-1))$  which maps $\sigma_i$ to $1$ and all other $p$-simplices to $0$. Since $\sigma_i$ is maximal, $\omega_i$ is a $p$-cocycle. 

We claim that $\sigma_i$ can't be a term in $\alpha_j$ when $j\neq i$. Otherwise, by Lemma~\ref{max_2n}, $\bigcup \sigma_i=\bigcup \sigma_j$ for some $i\neq j$ and then $S_i=S_j$. This is a contradiction since  $S_i$ and $S_j$ are different $2n$-subset of $[m]$ by the assumption. 

Consider $[\alpha_i]$ as a homology class in $H_p(\mathcal{VR}(\mathcal{F}^{[m]}_n; 2(n-1))$ and  $[\omega_j]$ as a cohomology class in $H^p(\mathcal{VR}(\mathcal{F}^{[m]}_n; 2(n-1))$. By the claim above and Lemma~\ref{same_coef_max}, the cap product $[\omega_i]^\frown[\alpha_j] =1$ if and only if $i=j$; otherwise $0$. 

 We claim that $\{[\alpha_i]: i=1, 2, \ldots, {m\choose 2n}\}$ is a linearly independent collection of generators in the homology group $H_p(\mathcal{VR}(\mathcal{F}_n^{[m]}; 2(n-1))$. Assume that $$\sum_i\lambda_i[\alpha_i]=0.$$ Then,  we apply the cocycle class $[\omega_j]$ via the cap product to the equation and obtain that $\lambda_j=0$ for any $j=1, 2, \ldots, {m\choose 2n}$. This finishes the proof. \end{proof}
\begin{table}
 \caption{Lower bounds of the rank of $p$-dimensional homomology of independence complex of the Kneser graph KG($n, k$) where $p = \frac{1}{2}{2n\choose n}-1$.}
\label{tab:largest_dim_rank}
\begin{tabular}{ |c|c|c|c|c|c| } 
 \hline
 \backslashbox[2mm]{$n$}{$k$} & 1 & 2& 3 &  4 & 5 \\ 

 \hline

 3 & 7  &  28  & 84 & 210 & 462\\
 \hline
 4 & 9 &  45  & 165 & 495 & 1,287\\
 \hline
 5 & 11 & 66  & 286 & 1,001 & 3,003\\
 \hline
 6 & 13 &  91  & 455 & 1,820 & 6,188\\
 \hline
 7 & 15 & 120  & 680 & 3,060 & 11,628\\
 \hline
 8 & 17 & 153  & 969 & 4,845 & 20,349\\
 \hline
 9 & 19 & 190  & 1,330 & 7,315 & 33,649\\ 
 \hline
 10 & 21 & 231  & 1,771 & 10,626 & 53,130\\ 
 \hline
\end{tabular}

\end{table}

\section{Concentration Maps}\label{centration}

In order to establish the lower bounds for other dimensional homologies, %of the complex Ind(KG$(n, k))$, 
we need to employ a contraction map from $\mathcal{F}^{S'}_n$ to $\mathcal{F}^{S}_n$ for subsets $S$ and $S'$ of $[m]$ with $S\subset S'$. Note that there is no natural projection map and so we introduce a stronger kind of contraction called concentration. Next we give its definition.    

\textbf{Concentration maps.} Fix $m$ and $n$  with $m\geq  n$. Choose $S'$ and $S$ to be subsets of $[m]$ such that: 1) $S'\supseteq S$; 2) $S$ has size $\geq n$.  We define a concentration map $\phi_S^{S'}$ from $\mathcal{F}_n^{S'}$ to $\mathcal{F}_n^{S}$ by the following: for each $A\in \mathcal{F}_n^{S'}$, let $\phi_S^{S'}(A)$ be the union of $A\cap S$ and the set of the smallest $|A\setminus S|$-many numbers in $S\setminus A$. So clearly $\phi^{S'}_S(A)=A$ for any $A\in \mathcal{F}^S_n$, i.e. the restriction of $\phi_S$ on $\mathcal{F}_n^S$ is the identity map.

The next result shows that the concentration map $\phi_S^{[m]}$ with $S\subset [m]$ can be represented as a composition atomic concentration maps, i.e. the concentration maps $\phi_S^{S'}$ with $|S'\setminus S|=1$ and $S\subset S'$. We skip the proof since it is straightforward to verify.

\begin{lemma}\label{decomp} Let $S$ be a subset of $[m]$ with $|S|\geq n$. Let $S_0, S_1, \ldots, S_k$ be a sequence of subsets of $[m]$ such that 
 $[m] =S_0$, $S_k = S$, and  $|S_i\setminus S_{i+1}|=1$. Then, $$\phi_S^{[m]} =\phi_{S_k}^{S_{k-1}}\circ \cdots \circ \phi_{S_1}^{S_0}.$$ 
\end{lemma}

 In order to prove that a general concentration map $\phi_S^{[m]}$ is a contraction, we start with proving the atomic concentration maps are contractions.   
\begin{lemma}
\label{contraction}Let $m$ and $n$ be integers with $m\geq n+1$. Let $S$ and $S'$ be subsets of $[m]$ such that $|S|\geq n$, $S'\supset S$, and $|S'\setminus S|=1$. Denote $S'\setminus S=\{\ell\}$. Then the mapping $\phi_S^{S'}$ from $\mathcal{F}_n^{S'}$ to $\mathcal{F}_n^{S}$ satisfying the following: 
\begin{itemize}
\item[1)] for each $A\in \mathcal{F}_n^{S'}$, $d(A, \phi_S^{S'}(A))$ is either $0$ or $2$;  

\item[2)] for any $A, B\in \mathcal{F}_n^{S'}$ with $A\neq B$, $d(\phi_S^{S'}(A),\phi_S^{S'}(B))$ is either $d(A, B)$ or $d(A, B)-2$. 

\end{itemize}
Hence, it is a contraction which induces a simplicial map from $\mathcal{VR}(\mathcal{F}^{S'}_n; r)$ to  $\mathcal{VR}(\mathcal{F}^S_n; r)$ for any scale $r\geq 0$.  \end{lemma}

\begin{proof} Take $A\in \mathcal{F}_n^{S'}$.  For each $A$ containing $\ell$, denote $i_A=\min\{i: i\in S\setminus A\}$; and clearly $i_A\neq \ell$. By the definition, $d(\phi_S^{S'}(A), A)=0$ if $\ell\notin A$. If $\ell \in A$, the symmetric difference of the sets $\phi_S^{S'}(A)$ and $A$ is $\{\ell, i_A\}$ which means that $d(\phi_S^{S'}(A), A) = 2$. Hence, part 1) holds.

Next we prove part 2). Take any pair $A, B\in \mathcal{F}_n^{[m]}$ with $A\neq B$.  We divide the proof into the following three cases.

\begin{itemize}
\item[i)] If $\ell\notin A\cup B$, then $d(\phi_S^{S'}(A), \phi_S^{S'}(B))=d(A, B)$ by the definition of $\phi_S^{S'}$.  

\item[ii)] Assume $\ell\in A\setminus B$, i.e. $\ell\in A$ but $\ell\notin B$. Then $\phi_S^{S'}(B)=B$. 

If $i_A\in B$, then $\phi_S^{S'}(A)\setminus B =(A\setminus B)\setminus \{\ell\}$ and $B\setminus \phi_S^{S'}(A) = (B\setminus A)\setminus \{i_A\}$ which means $d(\phi_S^{S'}(A), \phi_S^{S'}(B)) = d(A, B)-2$. 

If $i_A\notin B$, then $\phi_S^{S'}(A)\setminus B =((A\setminus B)\setminus \{\ell\})\cup \{i_A\}$ and $B\setminus \phi_S^{S'}(A) = B\setminus A$; hence in this case $d(\phi_S^{S'}(A), \phi_S^{S'}(B)) = d(A, B)$.    

\item[iii)] Assume that $\ell\in A\cap B$. Let $a =|A\cap B|$. Then $d(A, B) = 2(n-a)$. 

If $i_A=i_B$, then $\phi_S^{S'}(A)\Delta\phi_S^{S'}(B)=A\Delta B$; hence $d(\phi_S^{S'}(A), \phi_S^{S'}(B))=d(A,  B)$.

Now without loss of generality suppose that $i_A<i_B$. Then $i_A\in B\setminus A$ by definition of $i_B$. Notice that $i_A\neq \ell$. Then, $\phi_S^{S'}(B)\setminus \phi_S^{S'}(A) =((B\cup \{i_B\})\setminus ((A
\cup \{i_A\})\cap (B\cup\{i_B\}))$; and $\phi_S^{S'}(A)\setminus \phi_S^{S'}(B) = (A\cup \{i_A\})\setminus ((A\cup \{i_A\})\cap (B\cup\{i_B\}))$.  Then if $i_B\in A$, then $|\phi_S^{S'}(B)\setminus \phi_S^{S'}(A)| = n-a-1$ and $|\phi_S^{S'}(A)\setminus \phi_S^{S'}(B)| = n-a-1$ hence $d(\phi_S^{S'}(A), \phi_S^{S'}(B))= 2(n-a)-2 = d(A, B)-2$; otherwise,   $|\phi_S^{S'}(B)\setminus \phi_S^{S'}(A)| = n-a$ and $|\phi_S^{S'}(A)\setminus \phi_S^{S'}(B)| = n-a$, hence $d(\phi_S^{S'}(A), \phi_S^{S'}(B))= 2(n-a)=d(A, B)$.  \end{itemize}
\end{proof}

\begin{cor}\label{contraction_m_S}If $S$ is a subset of $[m]$ with $S\geq n$, $\phi_S^{[m]}$ is a contraction from $\mathcal{F}_n^{[m]}$ to $\mathcal{F}_n^{S}$ which induces a simplicial map from $\mathcal{VR}(\mathcal{F}^{[m]}_n; r)$ to  $\mathcal{VR}(\mathcal{F}^S_n; r)$ for any scale $r\geq 0$.\end{cor}

\begin{proof}We fix a sequence $\{S_0, S_1, \ldots, S_k\}$ of subsets of $[m]$ such that $[m] =S_0$, $S_k = S$, and  $|S_i\setminus S_{i+1}|=1$. By Lemma~\ref{decomp}, $\phi_S^{[m]} =\phi_{S_k}^{S_{k-1}}\circ \cdots \circ \phi_{S_1}^{S_0}$. Note that for $i =1, 2, \ldots k$ each  mapping $\phi_{S_i}^{S_{i-1}}$ is a contraction. So the mapping $\phi_S^{[m]}$ is a contraction since the composition of contraction mappings is a contraction also.    \end{proof}

Next, we prove more properties of the concentration map which allow us to provide a homology propagation result in the setting of the complex $\mathcal{VR}(\mathcal{F}_n^{[m]}; 2(n-1))$.
\begin{lemma}\label{phi_m_S} Let $S$ be a subset of $[m]$ with $|S|= \ell \geq n$ and $1\in S$. Let $T$ be an $\ell$-sized subset of $[m]$ with $T\neq S$. Then, the following results hold. 
\begin{itemize}
\item[1)] The restriction of the mapping $\phi_S^{[m]}$ on the metric space $\mathcal{F}_n^T$ is a bijective isometric mapping if  $1\notin T$ and $S\setminus \{1\}\subset T$. 

\item[2)] If $1\in T$ or $|T\setminus S|=1 $, then  $\mathcal{VR}(\phi_S^{[m]}(\mathcal{F}_n^T); 2c)$ is homotopy equivalent to $\mathcal{VR}(\mathcal{F}^{R}_n; 2c)$ with $R$ being a proper subset of $S$ and $c=n-1$. 

\item[3)] If $|T\setminus S|\geq 2$, then  $\mathcal{VR}(\phi_S^{[m]}(\mathcal{F}_n^T), 2c)$ is homotopy equivalent to $\mathcal{VR}(\mathcal{F}^{R}_n; 2c)$ with $R$ being a proper subset of $S$ and $c=n-1$.

\end{itemize}
\end{lemma}

\begin{proof} Pick an $\ell$-sized subset $T$ of $[m]$ with $T\neq S$ such that $1\notin T$ and $S\setminus \{1\}\subset T$. Let $i_0$ be the number in $T\setminus S$. Notice that $\phi_S^{[m]}(A) = A$ if $A$ be an $n$-subset of $S\cap T$ and $\phi_S^{[m]}(A)= (A\setminus \{i_0\})\cup \{1\}$ if $i_0\in A$. Then it is straightforward to verify that $\phi_S^{[m]}|_{\mathcal{F}_n^T}$ is a bijective isometric mapping. This finishes the proof of 1). 

For part 2) in the statement, we prove one special case and  all other cases follows from  a minor modification of the following argument. Let $T$ be an $\ell$-sized subset of $[m]$ with $1\in T$ and $|T\setminus S|=1$. We list $S=\{1, i_1, i_2, \ldots, i_{\ell-1}\}$ with increasing order and assume that $T=\{1, i_2, \ldots, i_{\ell-1}, i_{\ell}\}$. Let $R=S\cap T$. Fix $c$ between $1$ and $n$. We claim that $\phi_S^{[m]}(\mathcal{VR}(\mathcal{F}^T_n; 2c))$ is homotopy equivalent to $\mathcal{VR}(\mathcal{F}^R_n; 2c)$. Let $A$ be an $n$-subset of $T$. There are three cases for $\phi_S^{[m]}(A)$ as follows.   

\begin{itemize}
\item[i)] If $A\subset S$, clearly $\phi_S^{[m]}(A)=A\subset R$. 
\item[ii)] If $1\notin A$ and $i_\ell\in A$, then $\phi_S^{[m]}(A)=\{1\}\cup (A\setminus\{i_\ell\})\subset R$. 

\item[iii)] If $1\in A$ and $i_\ell\in A$, then $\phi_S^{[m]}(A)=\{i_1\}\cup (A\setminus\{i_\ell\})$

\end{itemize}

Note that only in case iii), vertices not in $\mathcal{F}^{R}_n$ are generated under the map $\phi^{[m]}_S$ and there are ${\ell-2\choose n-2}$-many such vertices in the complex $\phi_S^{[m]}(\mathcal{VR}(\mathcal{F}^T_n; 2c))$. Notice that each of such vertex contains both $1$ and $i_1$ and we list then as $B_1$, $B_2, \ldots, B_{\ell-2\choose n-2}$. Next we prove by induction adding these vertices in $\mathcal{VR}(\mathcal{F}^{R}_n; 2c)$ does change the homotopy type using Lemma~\ref{complex_add_1v}. Inductively assume that $\mathcal{VR}(\mathcal{F}^{R}_n\cup \{B_1, \ldots, B_{j-1}\}; 2c)$ is homotopy equivalent to  $\mathcal{VR}(\mathcal{F}^{R}_n; 2c)$. Denote $K= \mathcal{VR}(\mathcal{F}^{R}_n\cup \{B_1, \ldots, B_{j-1}, B_j\}; 2c)$. Pick $C\in 
\text{lk}_K(B_j)$ with $C\subset R$ and $B_j\setminus \{i_1\}\subset C$. Since $c=n-1$, $d(B_j, A)\leq 2c$ if and only if $B_j\cap A\neq \emptyset$. So $d(B_j, B_i)\leq 2c$ and $d(C, B_i)\leq 2c$ for $i=1, 2, \ldots, j-1$. Pick an $A\in\mathcal{F}_n^{R}$. By the definition of $C$, $A\cap B_j \subset C\cap B_j$. Therefore, $d(A, B_j)\leq 2c$ implies that $d(C\cap B_j)\leq 2c$. Hence the link of $B_j$ in $K$ is a cone with vertex $C$ which means that lk$_K(B_j)$ is contractible. By Lemma~\ref{complex_add_1v}, $K$ is homotopy equivalent to $K\setminus B_j$ which is homotopy equivalent to $\mathcal{VR}(\mathcal{F}^{R}_n; 2c)$ by the inductive assumption. This finishes the proof of part 2).     

Suppose that the size of $T\setminus S$ is $\geq 2$ and let $R=(T\cap S)\cup \{1\}$. So the size of $R$ is at most $\ell-1$ and $R$ is a proper subset of $S$ due to the fact that $S$ contains $1$. A similar argument as above can be applied to show that the complex $\mathcal{VR}(\phi_S^{[m]}(\mathcal{F}_n^T); 2c)$ is homotopy equivalent to $\mathcal{VR}(\mathcal{F}^{R}_n; 2c)$ when $c=n-1$. \end{proof}

\begin{lemma}\label{phi_m_m-1} Fix $m>\ell>n$. Let $S$ be an $\ell$-sized subset of $[m]$ containing $1$ and $m$ and $T=S\setminus \{m\}$. Then for $c=n-1$,  the image of the complex $\mathcal{VR}(\mathcal{F}^S_n; 2c)$ under the map $\phi_{[m-1]}^{[m]}$ is homotopy equivalent to $\mathcal{VR}(\mathcal{F}^T_n; 2c)$.   \end{lemma}

\begin{proof} Note that $\phi_{[m-1]}^{[m]} (\mathcal{VR}(\mathcal{F}^S_n; 2c)) =  \mathcal{VR}(\phi_{[m-1]}^{[m]}(\mathcal{F}^S_n); 2c)$. Since $T\subset S$ and $\phi_{[m-1]}^{[m]}(A)=A$ for each $n$-subset $A$ of $T$, $\mathcal{F}^T_n \subseteq \phi_{[m-1]}^{[m]}(\mathcal{F}^S_n)$. So if $\phi_{[m-1]}^{[m]}(A)\notin \mathcal{F}^T_n$, then $m\in A$ which implies that $1\in \phi_{[m-1]}^{[m]}(A)$. Therefore if $B\in \phi_{[m-1]}^{[m]}(\mathcal{F}^S_n)\setminus \mathcal{F}^T_n$, then $B$ contains $1$. 

We list the vertices in $\phi_{[m-1]}^{[m]}(\mathcal{F}^S_n)\setminus \mathcal{F}^T_n$ as $B_1, B_2, \ldots, B_b$ for some integer $b$. Note that $1\in \bigcap_{i=1}^bB_i$. We apply Lemma~\ref{complex_add_1v} inductively to show that  $$\mathcal{VR}(\phi_{[m-1]}^{[m]}(\mathcal{F}^S_n); 2c)\simeq \mathcal{VR}(\mathcal{F}^T_n; 2c).$$

Suppose that $\mathcal{VR}(\mathcal{F}^T_n\cup \{B_1, \ldots, B_{j-1}\}; 2c)\simeq \mathcal{VR}(\mathcal{F}^T_n; 2c)$. We define that $K=\mathcal{VR}(\mathcal{F}^T_n\cup \{B_1, \ldots, B_{j-1}, B_j\}; 2c)$. The pick $C$ in the link of $B_j$ in $K$ such that $(B_j\cap T)\subset C$. Such a $C$ exists because $\ell>n$ and  $|B_j\cap T|=n-1$. Since $c=n-1$, $d(A, B)\leq 2c$ if and only if $A\cap B\neq \emptyset$. Using the same argument in the proof of part 2) in Lemma~\ref{phi_m_S},  we obtain that  $d(C, A)\leq 2c$ for any $A\in \mathcal{F}^T_n\cup \{B_1, \ldots, B_{j-1}\}$ with $d(A, B_j)\leq 2c$. This means the link of $B_j$ in $K$ is a cone with vertex $C$ which is contractible. Then by Lemma~\ref{complex_add_1v}, $K$ is homotopy equivalent to $\mathcal{VR}(\mathcal{F}^T_n\cup \{B_1, \ldots, B_{j-1}\}; 2c)$ which is homotopy equivalent to $\mathcal{VR}(\mathcal{F}^T_n; 2c)$ by the inductive assumption. This finishes the proof.   \end{proof}
Combining the previous two lemmas, we get the following corollary.

\begin{cor}\label{phi_l+1_l}  Let $S=[\ell]$ and $T$ be an $\ell$-sized subset of $[\ell+1]$ with $T\neq S$. Then, 
\begin{itemize}
\item[1)] If $T=[\ell+1]\setminus \{1\}$, the restriction of the mapping $\phi_S^{[\ell+1]}$ on $\mathcal{F}_n^T$ is a bijective isometric mapping. 

\item[2)] For $T$ not in the form of $S$ or $T_S$, $\mathcal{VR}(\phi_S^{S'}(\mathcal{F}_n^T); 2(n-1))$ is homotopy equivalent to $\mathcal{VR}(\mathcal{F}^{S\setminus \{i_k\}}_n; 2(n-1))$ for some $k=2, 3, \ldots, n$. 
\end{itemize}

\end{cor}

\textbf{Persistent homology of $\mathcal{F}_n^{[m]}$.} Since $d(A, B)$ is even for each pair $A$, $B$ in $\mathcal{F}^{[m]}_n$, a simplex $\sigma \in \mathcal{VR}(\mathcal{F}^{[m]}_n; r)$ iff it is in $\mathcal{VR}(\mathcal{F}^{[m]}_n;r+1)$ when $r$ is even; hence when $r$ is even, $\mathcal{VR}(\mathcal{F}^{[m]}_n; r)=\mathcal{VR}(\mathcal{F}^{[m]}_n; r+1)$, i.e., the inclusion map is the identity map. The next result shows that the inclusion map from $\mathcal{VR}(\mathcal{F}^{[m]}_k, r)\hookrightarrow \mathcal{VR}(\mathcal{F}^{[m]}_k; r+1)$ is homotopically trivial when $r$ is odd.  This means that, in the setting of Vietoris–Rips complexes of $\mathcal{F}_n^{[m]}$, persistent homology does not provide any new information beyond the homology groups at fixed scales.

\begin{theorem} For any positive integers $n$, $m$, and $c$ with $n\leq m$, the natural inclusion $\mathcal{VR}(\mathcal{F}^{[m]}_n; 2c)\hookrightarrow \mathcal{VR}(\mathcal{F}^{[m]}_n; 2(c+1))$ is homotopically trivial.  \end{theorem}

\begin{proof} Notice that the concentration map $\phi^{[m]}_{[m]}$ is the identity map on $\mathcal{F}^{[m]}_n$ and hence it  induces the natural inclusion map $\imath: \mathcal{VR}(\mathcal{F}^{[m]}_n; 2c)\hookrightarrow \mathcal{VR}(\mathcal{F}^{[m]}_n; 2(c+1))$. 

Next we claim that for each $j=n, n+1, \ldots, m-1$, the concentration maps $\phi_{[j+1]}^{[m]}$ and $\phi_{[j]}^{[m]}$ are homotopic to each other in $\mathcal{VR}(\mathcal{F}^{[m]}_n; 2(c+1))$. To prove this claim, it is sufficient to prove they are contiguous in $\mathcal{VR}(\mathcal{F}^{[m]}_n; 2(c+1))$. 

Fix $j\geq n$. Pick a simplex $\sigma$ in  $\mathcal{VR}(\mathcal{F}^{[m]}_n; 2c)$. We show that $\phi_{[j+1]}^{[m]}(\sigma)\cup \phi_{[j]}^{[m]}(\sigma)$ is a simplex in $\mathcal{VR}(\mathcal{F}^{[m]}_n; 2(c+1))$, i.e. $d(A, B)\leq 2(c+1)$ for any two vertices $A, B$ in $\phi_{[j+1]}^{[m]}(\sigma)\cup \phi_{[j]}^{[m]}(\sigma)$.  Take two vertices $A$ and $B$ in $\phi_{[j+1]}^{[m]}(\sigma)\cup \phi_{[j]}^{[m]}(\sigma)$. First suppose that both $A$ and $B$ are in $\phi_{[j+1]}^{[m]}(\sigma)$. We pick $A', B'$ in $\sigma$ with $\phi_{[j+1]}^{[m]}(A')=A$ and $\phi_{[j+1]}^{[m]}(B')=B$. By Corollary~\ref{contraction_m_S}, we obtain that 

$$d(A, B)=d(\phi_{[j+1]}^{[m]}(A'),\phi_{[j+1]}^{[m]}(A'))\leq d(A', B')\leq 2c.$$

Similarly the result holds if both both $A$ and $B$ are in $\phi_{[j]}^{[m]}(\sigma)$. Now we assume that $A\in \phi_{[j+1]}^{[m]}(\sigma)$ and $B\in \phi_{[j]}^{[m]}(\sigma)$. Take $A', B'$ in $\sigma$ with $\phi_{[j+1]}^{[m]}(A')=A$ and $\phi_{[j]}^{[m]}(B')=B$. Note that $\phi_{[j]}^{[m]} =\phi_{[j]}^{[j+1]}\circ \phi_{[j+1]}^{[m]}$ by Lemma~\ref{decomp}. Then by Corollary~\ref{contraction_m_S} and part 1) in Lemma~\ref{contraction}, we obtain that  
$$d(A, B)\leq  d(\phi_{[j+1]}^{[m]}(A'), \phi_{[j+1]}^{[m]}(B'))+ d(\phi_{[j+1]}^{[m]}(B'), \phi_{[j]}^{[j+1]}\circ\phi_{[j+1]}^{[m]}(B')))\leq d(A', B')+2.$$

Hence $d(A, B)\leq 2(c+1)$ since $A', B'\in \sigma$ which is a simplex in $\mathcal{VR}(\mathcal{F}^{[m]}_2; 2c)$. This proves the claim. By the claim, we obtain a sequence of homotopic mappings as follows: 

$$\imath =\phi_{[m]}^{[m]}\simeq \phi_{[m-1]}^{[m]}\simeq \cdots \simeq \phi_{[n]}^{[m]}.$$

Notice the last mapping $\phi_{[n]}^{[m]}$ is a constant map. And therefore the natural inclusion map $\imath$ is homotopically trivial.  \end{proof}

\section{Homology propagation via concentration maps}\label{homology_p}

In Section~\ref{max_cpgenerators}, the cross-polytopal homology generators are the key to identify non-trivial homology and some lower bounds in the Vietoris-Rips complexes on $\mathcal{F}_n^{[m]}$. In this section, with the help of concentration maps, we aim to develop a new approach for providing lower bounds of any homologies without any knowledge of homology generators. Similar results in the setting of Vietoris-Rips complex of hypercube graphs were obtained by Adams and Virk in \cite{AV24}.    

 Suppose $\ell$ is the smallest number such that $\tilde{H}_q(\mathcal{VR}(\mathcal{F}^{[\ell]}_n;2(n-1)))$ is non-trivial. We consider the complexes, $L=\mathcal{VR}(\mathcal{F}^{[\ell]}_n;2(n-1))$ and $K=\mathcal{VR}(\mathcal{F}^{[\ell+1]}_n;2(n-1))$. Note that there are $(\ell+1)$-many natural isometric embeddings of $L$ in $K$ and there is a concentration map from $K$ to each of the embeddings. Therefore by Lemma~\ref{contraction_ind_inj}, the inclusion maps from the embeddings of $L$ to $K$ induces injective homomorphism in all the homology groups. In the following arguments, we use the same notations for the homology classes in the embedding of $L$ and its image in the homology group of $K$.  The following result states that in $\tilde{H}_q(K)$, all but at most one of the embedding of $L$ induce independent inclusions on the $q$th homology group. 

\begin{theorem}\label{codim1} Fix $q\geq 1$. If $\ell$ is the smallest integer such that $\tilde{H}_q(\mathcal{VR}(\mathcal{F}_n^{[\ell]}; 2(n-1))$ is nontrivial, then 
$$\text{rank} (\tilde{H}_q(\mathcal{VR}(\mathcal{F}_n^{[\ell+1]}; 2(n-1)))\geq \ell\cdot \text{rank} (\tilde{H}_q(\mathcal{VR}(\mathcal{F}_n^{[\ell]}; 2(n-1))). $$
    
\end{theorem}

\begin{proof} The metric space $\mathcal{F}^{[\ell+1]}_n$ contains $(\ell+1)$-many isometric copies of of $\mathcal{F}^{\ell}_n$. We enumerate these copies as $C_1, C_2, \ldots, C_{\ell+1}$ such that each $C_j$ is $\mathcal{F}^{[\ell+1]\setminus \{\ell+2-j\}}_n$. We see that $C_1  = \mathcal{F}^{[\ell]}_n$ and $C_{\ell+1} = \mathcal{F}^{[\ell+1]\setminus \{1\}}_n$.  For each $j=1, 2, \ldots, \ell$, we pick a collection of linearly independent generators $\{g_{i, j}: 1\leq i\leq \text{rank}(\tilde{H}(\mathcal{VR}(\mathcal{F}_n^{[\ell]}, )))\}$ in $\tilde{H}(\mathcal{VR}(C_j; 2(n-1)))$.  

We claim that the collection $\{g_{i, j}: 1\leq i\leq \text{rank}(\tilde{H}(\mathcal{VR}(\mathcal{F}_n^{[\ell]}; 2(n-1)))), 1\leq j\leq \ell\}$ is linearly independent. 

Assume that $$\sum_{i, j}\lambda_{i, j}\cdot g_{i, j}=0$$

for $\lambda_{i, j}$ with $1\leq i\leq \text{rank}\tilde{H}_q(\mathcal{VR}(\mathcal{F}_n^{[\ell]}; 2(n-1)))$ and $1\leq j\leq \ell$. We show that $\lambda_{i, 1}= 0$ for $1\leq i\leq \text{rank}(\tilde{H}(\mathcal{VR}(\mathcal{F}_n^{[\ell]}; 2(n-1))))\}$. Similar discussion can be used to show that $\lambda_{i, j}=0$ for $j\geq 2$. Let $f$ be the concentration map $\phi_{[\ell]}^{[\ell+1]}$ from $\mathcal{F}^{\ell+1}_n$ to $\mathcal{F}^{\ell}_n$. By Corollary~\ref{phi_l+1_l}, we get the following conditions: 
\begin{itemize}
\item[i)] $f$ is the identity map on $C_1$ and is a bijective isometric map on $C_{\ell+1}$;
\item[ii)] $f$ maps all other $C_j$ with $2\leq j\leq \ell$ to a subcomplex in $\mathcal{VR}(\mathcal{F}_{n}^{[\ell+1]}; 2(n-1))$ which is homotopy equivalent to $\mathcal{VR}(\mathcal{F}_n^{[\ell-1]}; 2(n-1))$ whose $q$-dimensional homology is trivial.  
\end{itemize}
Then we apply the induced homology homomorphism $f_\ast$ on the generators and obtain that $f_\ast (g_{i, 1}) = g_{i, 1}$ and $f_\ast (g_{i, j}) = 0$ for $2\leq j\leq \ell$. Hence we obtain that  $$\sum_{i}\lambda_{i, 1}g_{i, 1}=0.$$
Consequently, the coefficients $\lambda_{i, 1}=0$ for $i=1, 2,\ldots, \text{rank}(\tilde{H}(\mathcal{VR}(\mathcal{F}_n^\ell; 2(n-1)))$ due the linear independence of the generators. And this finishes the proof. 
\end{proof}

Next we generalize the arguments in Theorem~\ref{codim1} to the complex $\mathcal{VR}(\mathcal{F}_n^{[m]}; 2(n-1))$ via concentration maps using Lemmas~\ref{phi_m_m-1} and~\ref{phi_m_S}. In particular cases, when $m=\ell$, Theorem~\ref{low_b_smalldim} trivially holds; and when $m=\ell+1$, we recover Theorem~\ref{codim1}.  

\begin{theorem}\label{low_b_smalldim}  Fix $q\geq 1$. If $\ell$ is the smallest integer such that $\tilde{H}_q(\mathcal{VR}(\mathcal{F}_n^{[\ell]}; 2(n-1))$ is nontrivial, then for any $m\geq \ell$ 
$$\text{rank} (\tilde{H}_q(\mathcal{VR}(\mathcal{F}_n^{[m]}; 2(n-1)))\geq \sum_{i=\ell}^m{i-2\choose \ell-2}\cdot \text{rank} (\tilde{H}_q(\mathcal{VR}(\mathcal{F}_n^{[\ell]}; 2(n-1))). $$

\end{theorem}
\begin{proof} We consider $\mathcal{F}_n^S$ for $S$ being a $\ell$-subset of $[m]$ with $m\in S$. Clearly there are ${m-1\choose \ell-1}$-many such copies in which ${m-2\choose \ell-1}$-many of them don't contain $1$. We list all the $\ell$-subsets of $[m]$ containing both $1$ and $m$ as $S_1, S_2, \ldots, S_{a}$ with $a={m-1\choose \ell-1}-{m-2\choose \ell-1} ={m-2\choose \ell-2}$. For each $1\leq j\leq a$, we pick a collection of linearly independent generators $\{g_{i, j}: 1\leq i\leq \text{rank}(\tilde{H}_q(\mathcal{VR}(\mathcal{F}_n^{[\ell]}; 2(n-1)))\}$ in homology group $\tilde{H}_q(\mathcal{VR}(S_j; 2(n-1)))$.

Also, by induction we fix a collection of linearlly independant generators $$\{h_j: 1\leq j\leq \sum_{i=\ell}^{m-1} {i-2\choose \ell-2}\}$$ in $\mathcal{VR}(\mathcal{F}_n^{[m-1]}, 2n-2)$. We show that the generators  in  collection  $\{g_{i, j}: 1\leq \text{rank}(\tilde{H}(\mathcal{VR}(\mathcal{F}_n^{[\ell]}, 2n-2)), 1\leq j\leq a\}\cup \{h_j: 1\leq j\leq \sum_{i=\ell}^{m-1} {i-2\choose \ell-2}\}$ are linearly independent. 

Suppose that $$\sum_{i, j}\lambda_{i, j}g_{i, j}+\sum_{j} \mu_j h_j=0$$ for real numbers $\lambda_{i, j}$ and $\mu_j$. Consider the concentration map $\phi_{[m-1]}^{[m]}$. By Lemma~\ref{phi_m_m-1}, if $S$ is an $\ell$-sized subset of $[m]$ containing  $1$ and $m$, then the images of  $\mathcal{VR}(\mathcal{F}^{S}_n; 2n-2)$ under $\phi_{[m-1]}^{[m]}$ is  homotopy equivalent to some $\mathcal{VR}(\mathcal{F}^{R}_n; 2n-2)$ with $R$ being a proper subset of $S$. Note each of such $R$ has size $<\ell$, therefore their $q$-dimensional homology is trivial. Then we apply the induced homology homomorphism $(\phi_{[m-1]}^{[m]})_\ast$ to equation and we obtain that $\sum_{j}\mu_{j} h_{j}=0$ which yields that $\mu_{j}=0$ for any $j$ due to the linear independence of the generators $h_{j}$'s.   

 Next, we show that $\lambda_{i, 1}=0$ for all $i=1, 2, \ldots, \text{rank}(\tilde{H}_q(\mathcal{VR}(\mathcal{F}_n^{[\ell]}; 2n-2))$. Same approach can be applied to obtain $\lambda_{i, j}=0$ for all $i, j$. Consider the concentration map $\phi^{[m]}_{S_1}$. By part 2) in Lemma~\ref{phi_m_S}, the induced homomorphism $(\phi^{[m]}_{S_1})_\ast$ maps $\mathcal{VR}(\mathcal{F}^{S}_n; 2n-2)$ is homotopy equivalent to some $\mathcal{VR}(\mathcal{F}^{S'}_n; 2n-2)$ with $|S'|<\ell$ if an $\ell$-sized subset $S$ of $[m]$ contains $1$. Then we apply the homomorphism $(\phi^{[m]}_{S_1})_\ast$ to the equation  $\sum_{i, j}\lambda_{i, j}g_{i, j}=0$  and obtain that $\sum_{i}\lambda_{i, 1}g_{i, 1}=0$. The linear independence of $\{g_{i, 1}: 1\leq i\leq \text{rank}(\tilde{H}_q(\mathcal{VR}(\mathcal{F}_n^{[\ell]}; 2n-2)) \}$ implies that $\lambda_{i, 1}=0$ for each $i$ which finishes the proof. \end{proof}

 Table~\ref{tab:3_k_KG_homology} provides several examples of lower bounds for the ranks of $9$th dimensional homology and $6$th dimensional homology of the independence complex Ind(KG$(3, k)$). The lower bounds for the ranks of $9$th dimensional homology is obtained using Theorem~\ref{low_b_bigdim}; and the other ones are from Theorem~\ref{low_b_smalldim}. Notice that the computed results for $6$th dimensional homology is smaller than the computed results in Table~\ref{KG_3_homology} when $k=2$ and $3$. 

 \begin{table}
 \caption{Lower bounds of the rank of nontrivial homomologies of the independence complex of Kneser graph KG($3, k$) using  Theorem~\ref{low_b_bigdim}, Theorem~\ref{low_b_smalldim}, and the result in Table~\ref{KG_3_homology}.}
\label{tab:3_k_KG_homology}
\begin{tabular}{ |c|p{0.6cm}|p{0.7cm}|p{0.7cm}|p{0.7cm}|p{0.7cm}| } 
 \hline
 \backslashbox[2mm]{homology}{$k$} & 1 & 2& 3 &  4 & 5 \\ 
  \hline
 6th-dim & 29   &  203  & 812 & 1972 &5626\\
 \hline
 9th-dim & 7  &  28  & 84 & 210 & 462\\
 \hline

\end{tabular}
\end{table}

\section{Connectivity}\label{connect}

In this section, we discuss the connectivity of the independence of Kneser graph using its total dominating number. We start with some fundamental definitions and results. 

Let $G$ be a graph with no isolated vertices.  The total dominating number of $G$, denoted by $\gamma_t(G)$, is the minimal size of total dominating sets of $G$. A subset $S$ of the vertices of $G$ is said to be a total dominating set if every vertex is connected by an edge to a vertex in $S$. 

It is proved in \cite{M03} that if $\gamma_t(G)>2k$, then Ind$(G)$ is $(k-1)$-connected. Also, the following result from \cite{HY13} is an immediate consequence by the definition of total dominating set if the order of $G$ is $m$ and $G$ has no isolated vertex: $$\gamma_t(G)\geq \frac{m}{\Delta (G)}$$
where $\Delta(G)$ is the maximal order of any vertex in $G$. Recall that the order of a graph is the number of vertices of the graph, and the order of a vertex in a graph is the number of edges meeting at the vertex. 

\begin{theorem}\label{connectivity}For a Kneser graph KG$(n, k)$, $$\gamma_t(\text{KG}(n, k))\geq \frac{{2n+k\choose n}}{{n+k \choose n}}.$$ 

Let $\alpha_{n, k} =\frac{1}{2}\cdot \frac{{2n+k\choose n}}{{n+k \choose n}}$. Then if $\alpha_{n, k}$ is an integer, the complex Ind(KG$(n, k)$) is $(\alpha_{n, k}-2)$-connected; otherwise it is $(\llcorner \alpha_{n, k} \lrcorner-1)$-connected.  
\end{theorem}

\begin{proof} The order of the Kneser graph KG$(n, k)$ is ${2n+k\choose n}$ and the order of any vertex in KG$(n, k)$ is ${n+k\choose n}$ by the definition of Kneser graph. Hence, $$\gamma_t(\text{KG}(n, k))\geq 2\alpha_{n, k}.$$
Therefore, when $\alpha_{n, k}$ is an integer, $\gamma_t(\text{KG}(n, k))> 2(\alpha_{n, k}-1)$ and the independence complex Ind(KG$(n, k)$) is $(\alpha_{n, k}-2)$-connected. If $\alpha_{n, k}$ is not an integer, then  $\gamma_t(\text{KG}(n, k))> 2\llcorner \alpha_{n, k} \lrcorner$ and  Ind(KG$(n, k)$) is $(\llcorner \alpha_{n, k} \lrcorner-1)$-connected. \end{proof}

In Table~\ref{tab:connectivity}, we list some examples of connectivity of the indepndence complex of Kneser graph obtained using Theorem~\ref{connectivity}.

\begin{table}
\begin{center}

\caption{Lower bounds of the connectivity of the independence complex of Kneser graph KG($n, k$).}
\label{tab:connectivity}
\begin{tabular}{ |c|c|c|c|c|c| } 
 \hline
 \backslashbox[2mm]{$n$}{$k$} & 1 & 2& 3 &  4 & 5 \\ 
 
 \hline
 4 & 11 &  5  & 3 & 2 & 1\\
 \hline
 5 & 37 & 17  & 9 & 6 & 4\\
 \hline
 6 & 121 &  52  & 28 & 17 & 11\\
 \hline
 7 & 400 & 157  & 79 & 46 & 30\\
 \hline
 8 & 1,349 & 484  & 227 & 125 & 77\\
 \hline
 9 & 4,617 & 1,525  & 666 & 346 & 202\\ 
  \hline
 10 & 16,031 & 4,897  & 1,998 & 978 & 542\\ 
 \hline
\end{tabular}

\end{center}
\end{table}
\section{Open questions}\label{open_p}

Still, little is known about the topological property of the complex Ind(KG$(n, k))$ for $n\geq 3$ and $k\geq 1$. Let $p = \frac{1}{2}{2n\choose n}-1$. Recall that the lower bounds obtained from Theorem~\ref{low_b_bigdim} for the rank of $p$-dimensional homology match the computed results when $n=3$ and $k=1, 2, 3$. Hence we conjecture that the lower bounds from Theorem~\ref{low_b_bigdim} are actually the exact rank of $p$-dimensional homology.  

\begin{conj} Let $n\geq 3$ and $p = \frac{1}{2}{2n\choose n}-1$. The rank of the $p$-dimensional homology group of Ind(KG$(n, k))$ is ${2n+k\choose 2n}$. \end{conj}

A general question in this subject is to determine the homotopy types of the independence complexes. 

\begin{ques} What is the homotopy type of the independence complex Ind(KG$(n, k)$) for $n\geq 3$ and $k
\geq 1$? Specifically, is Ind(KG$(3, k)$) homotopy equivalent to a wedge sum of $S^6$'s and $S^9$'s?\end{ques} 

It seems that the homotopy types of such independence complexes are very hard to determine. The following question is natural to ask. 

\begin{ques} What are the dimensions of non-trivial homology of the complex Ind(KG$(n, k)$) with $n\geq 3$ and $k\geq 1$? Is it true that $\tilde{H}_i(\text{Ind(KG}(n, k))\neq 0$ if and only if $i\in \{6, 9\}$?\end{ques}

Notice that in the discussion of Section~\ref{max_cpgenerators}, the homology generators in the complex $\text{Ind(KG}(n, k))$ are cross-polytopal subcomplexes containing a maximal simplex in $\text{Ind(KG}(n, k))$ with $n\geq 3$. 

\begin{ques} Fix $n\geq 3$. Is it true that all homology generators in the complex $\text{Ind(KG}(n, k))$ can be represented as a cycle in one of its cross-polytopal subcomplexes?  \end{ques}

\textbf{Acknowledgments}

\medskip

We would like to thank Henry Adams, Joseph Briggs, and \v{Z}iga Virk for helpful discussions.  This work was completed in part with resources provided by the Auburn University Easley Cluster.

\end{document}